\documentclass[11pt,a4paper]{article}
\usepackage[paper=a4paper, textwidth=16.3cm, textheight=24.0cm]{geometry}

\usepackage{graphicx}
\usepackage{amsmath,amsfonts,amssymb}
\usepackage{xcolor}
\usepackage{url}

\usepackage{enumitem}
\setlist[enumerate]{leftmargin=.5in}
\setlist[itemize]{leftmargin=.5in}

\usepackage{amsopn}

\newtheorem{theorem}{Theorem}[section]

\newtheorem{lemma}[theorem]{Lemma}

\begin{document}

\begin{center}
  {\LARGE \textbf{A distributed-delay model for the El Ni{\~n}o Southern Oscillation with a minimum delay: a case study of the shifted linear chain trick}}

\vspace*{5mm}

{\large {Joe Steele$^1$, Andrew Keane$^2$ and Bernd Krauskopf$^1$}} \\[3mm]

$^1$Department of Mathematics, The University of Auckland, 
Private Bag 92019, \\ Auckland 1142, New Zealand \\[1mm]

$^2$School of Mathematical Sciences and Sustainability Institute,
University College Cork, Ireland\\[5mm]

September 2026 \\[8mm]

\end{center}

%%%%%%%%%%%%%%%%%%%%%%%%%%%%%%%%%%%%%%%%%%%%%%%%%%%%%%%

\begin{abstract}
When modelling a system of interest with a delay differential equation, a distributed delay may be the appropriate modelling choice when the delayed response occurs over a significant range of times rather than with a single constant delay.
In systems with a minimum physical transit or processing time, however, the delay kernel should respect a fixed and positive minimum delay. A shifted Erlang kernel is a convenient choice, because the \emph{shifted linear chain trick} allows one to replace the convolution associated with this type of kernel by a finite auxiliary-variable representation with one constant delay. As we demonstrate with a case study of the Ghil-Zaliapin-Thompson (GZT) model of the El Ni{\~n}o Southern Oscillation with seasonal forcing and distributed delayed oceanic feedback, this reduction makes it possible to perform a bifurcation analysis with numerical continuation techniques for any width of the Erlang distribution. Specifically, we present the bifurcation and resonance structure of the distributed-delay GZT model in the plane of forcing strength and delay for different widths of the distribution. To ensure a like-for-like comparison and determine the influence of delay distribution, we present these results in terms of the effective delay and, moreover, rescale the feedback strength to account for its attenuation with increasing width of the distribution. 
\end{abstract}

%%%%%%%%%%%%%%%%%%%%%%%%%%%%%%%%%%%%%%%%%%%%%%%%%%%%%%
%                   6. BODY
%%%%%%%%%%%%%%%%%%%%%%%%%%%%%%%%%%%%%%%%%%%%%%%%%%%%%%

\section{Introduction}
Delays occur in many natural systems, and it is often useful to model their effects with delay differential equations (DDEs) \cite{erneux2009applied, Stepan1989}. As a modelling assumption and as a first step in an analysis, one often assumes delays to be constant. In many applications, however, the response due to a delayed effect or feedback is not concentrated at a single time.
This is the case, for example, in particular models of drug absorption, where a dose may pass through a sequence of transit compartments before entering systemic circulation \cite{SavicKarlsson2007}, and in biological models with distributed maturation or passage times \cite{cassidy2018recipe, HearnHaurieMackey1998}. The issue of non-constant delay times is also relevant to the El Ni{\~n}o Southern Oscillation (ENSO) phenomenon, where the oceanic response is associated with wave propagation and reflection across the equatorial Pacific Ocean \cite{ghil2008delay, wang2016nino}.

In such settings, the response is more naturally described by a distribution of response times rather than by a single constant delay. This leads to distributed-delay DDEs, which are often more faithful to the underlying mechanism but are also harder to analyse. Distributed DDEs are not readily compatible with standard numerical continuation and bifurcation software, because their convolution terms usually require direct approximation or truncation over a finite memory horizon. Moreover, their formulation involves integrals of a delay kernel over the memory, which need to be computed numerically in the general case.
This makes it difficult to carry out a systematic continuation-based analysis of the bifurcation structure for a distributed-delay DDE.

A classical way to overcome these difficulties is to identify delay kernels for which the distributed-delay term admits a finite-dimensional representation. The best-known version of such a reduction in the applied literature involves Erlang kernels, which are gamma distributions of integer order. This is referred to as the \emph{linear chain trick} (LCT), a term popularised by MacDonald \cite{MacDonald1987, macdonald1989biological}. The LCT replaces the convolution term exactly by the final component of a finite chain of auxiliary variables governed by ordinary differential equations (ODEs). MacDonald also emphasised the modelling interpretation of the LCT, in which an Erlang-distributed delay arises from a sequence of exponentially distributed transit or maturation stages, represented by the auxiliary variables in the chain.
This interpretation has since been widely used in biological and physiological modelling; further discussion of the LCT and related extensions can be found in \cite{cassidy2021distributed, DiekmannGyllenbergMetz2017, Fargue1974, hurtado2019generalizations, smith2011introduction}.

Despite these advantages, standard Erlang kernels have an important modelling limitation: they assign weight to arbitrarily small delays. This is inappropriate in systems where a minimum processing, transit, or maturation time is known to be present. In such cases, the delay distribution should not only describe the spread of possible response times, but also respect the lower bound on when a response can first occur. This can be achieved by translating the Erlang kernel along the time axis. The resulting shifted Erlang kernel vanishes on $[0,\tau_{\mathrm{m}})$, where $\tau_{\mathrm{m}}$ is the minimum delay or dwell time.

Shifted Erlang kernels have been used in models with minimum maturation or transit times \cite{CampbellJessop2009, HearnHaurieMackey1998}, and in stability studies of distributed delay equations \cite{BernardBelairMackey2001}.
Unlike standard Erlang kernels, however, shifted Erlang kernels do not reduce distributed-delay DDEs to purely ODE systems. The appropriate construction is what we refer to as the \emph{shifted linear chain trick} (SLCT), derived explicitly by Blythe in the context of a specific maturation model \cite{blythe1984dynamics}.
It reformulates the distributed-delay DDE as a finite chain of auxiliary ODEs plus a DDE with one constant delay $\tau_{\mathrm{m}}$.

The SLCT retains the main computational benefit of the classical LCT while allowing a physically meaningful minimum delay. The transformed system remains a DDE, but one with a single, constant delay, which is directly compatible with established continuation and bifurcation software, such as \texttt{DDE-BifTool} \cite{SIE14}. This makes the SLCT a useful bridge between physically motivated distributed-delay modelling and continuation-based dynamical systems analysis.

In this paper, we formulate the SLCT reduction for a general distributed-delay DDE and apply this framework to the Ghil-Zaliapin-Thompson (GZT) model of ENSO. This model is a natural test case because its delayed oceanic feedback represents a wave-transit process \cite{battisti1989interannual,ghil2008delay}. It is therefore reasonable to replace the original constant delay by a distributed delay, while also retaining a non-zero minimum transit time. First, in Section~\ref{sec:reduction}, we state and prove the SLCT for a general shifted Erlang-distributed DDE, extending the model-specific construction previously given by Blythe. In Section~\ref{sec:ENSO}, we use the SLCT to obtain an auxiliary-variable formulation of the distributed-delay GZT ENSO model that is directly suitable for numerical continuation; here the delayed feedback is approximated by a shifted Erlang kernel constructed via empirical estimates of oceanic wave speeds.
We first show that the SLCT for shifted distributed delays yields bifurcation diagrams that cannot be compared directly for different widths of the distribution. We then introduce suitable parameter rescalings, based on the effective delay and the attenuation of the delayed feedback, that allow for an effective comparison, including with the case of constant delay. Concluding remarks are provided in Section~\ref{sec:conclusion}.

\newpage

%%%%%%%%%%%%%%%%%%%%%%%%%%%%%%%%%%%%%%%%%%%%%%%%%%%%
\section{Reduction of distributed DDEs with a minimum delay}
\label{sec:reduction}

We now introduce the shifted distributed-delay formulation by first defining the shifted Erlang kernels and their key parameters, and then deriving the corresponding SLCT reduction.

%%%%%%%%%%%%%%%%%%%%%%%%%%%%%%%%%%%%%%%%%%%%%%%%%%%%
\subsection{Shifted Erlang kernels}
\label{subsec:erlang_kernels}

An $n^{\text{th}}$-order shifted Erlang distribution with scale parameter $\Delta>0$ and minimum delay $\tau_{\mathrm{m}}\ge0$ is given by
\begin{equation}
\label{eq:Erlang_shifted}
g^n(s) =
\left\{
\begin{aligned}
& \qquad\qquad\qquad0, \;  & \  \ 0\leq s < \tau_{\mathrm{m}}, \\
& \dfrac{(s-\tau_{\mathrm{m}})^{n-1} \exp\left(-\frac{s-\tau_{\mathrm{m}}}{\Delta}\right)}{\Delta^n(n-1)!}, \;  & s\geq\tau_{\mathrm{m}}. \ \ \,\\
\end{aligned}
\right.
\end{equation}

The minimum delay $\tau_{\mathrm{m}}$ fixes the earliest time at which the past state can influence the present. The scale parameter $\Delta$ controls the width of the distribution and the order $n$ its shape. The mean delay, also referred to as the effective delay of the shifted Erlang kernel, is
\begin{equation}
\label{eq:taueff}
    \tau_{\mathrm{eff}}=\tau_{\mathrm{m}}+n\Delta.
\end{equation}
Thus, changing $\Delta$ at fixed $\tau_{\mathrm{m}}$ and $n$ changes not only the width of the distribution, but also $\tau_{\mathrm{eff}}$.

Figure~\ref{fig:delay_kernels} illustrates this for $n=1,2,3,4$, with the minimum delay fixed at $\tau_{\mathrm{m}}=0.25$ in all panels. For each order, we show the narrow-kernel case $\Delta=10^{-3}$, for which $\tau_{\mathrm{eff}}\approx\tau_{\mathrm{m}}$ and the shifted distributed delay is effectively a constant delay. We also show two wider kernels chosen so that the effective delays are $\tau_{\mathrm{eff}}=\tau_{\mathrm{m}}+0.2$ and $\tau_{\mathrm{eff}}=\tau_{\mathrm{m}}+0.4$ in every panel; this allows us to compare the effect on the kernel shape of increasing the scale parameter $\Delta$ across different orders $n$ with the same $\tau_{\mathrm{eff}}$. Conversely, if $\tau_{\mathrm{eff}}$ is held fixed while $\Delta$ is increased, then the minimum delay must decrease. Physical admissibility requires $\tau_{\mathrm{m}}\ge0$, or equivalently $\tau_{\mathrm{eff}}\ge n\Delta$. This requirement is important in the computations that follow, where changes in the width of the delay distribution via the scale parameter $\Delta$ must be separated from the accompanying shift in $\tau_{\mathrm{eff}}$.

%%%%%%%%%%%%%%%%%%%%%%%%%%%%%%%%%%%%%%%%%%%%%%%%%%%%
\begin{figure}[t]
  \hspace*{10mm}
  \includegraphics[scale=1.11]{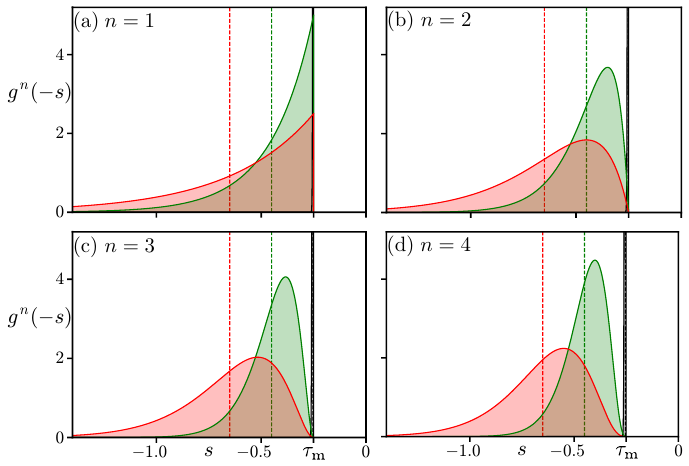}
    \caption{Shifted Erlang distributions $g^n(-s)$ given by~\eqref{eq:Erlang_shifted} with fixed minimum delay $\tau_{\mathrm{m}}=0.25$ for $n=1,2,3,4$ in panels (a)--(d), respectively. The black, green, and red curves are for increasing values of $\Delta$ that all give $\tau_{\mathrm{eff}} \approx \tau_{\mathrm{m}}$, $\tau_{\mathrm{eff}}=\tau_{\mathrm{m}}+0.2$ and $\tau_{\mathrm{eff}}=\tau_{\mathrm{m}}+0.4$ in \eqref{eq:taueff}, respectively, as indicated by the dashed vertical lines of the same colour.}
    \label{fig:delay_kernels}
\end{figure}
%%%%%%%%%%%%%%%%%%%%%%%%%%%%%%%%%%%%%%%%%%%%%%%%%%%%

From a modelling perspective, rather than being prescribed directly, $\tau_{\mathrm{m}}$, $\Delta$, and $n$ should be fitted to a delay distribution obtained from physical data or a mechanistic calculation. When such a fit can be obtained, the shifted Erlang kernel provides a compact and convenient approximation of the estimated distribution.
Section~\ref{subsec:gzt_models} illustrates this approach for the distributed feedback delay in ENSO.

%%%%%%%%%%%%%%%%%%%%%%%%%%%%%%%%%%%%%%%%%%%%%%%%%%%%
\subsection{The shifted linear chain trick}
\label{subsec:shiftedLCT}
We state the reduction for a single shifted Erlang kernel and a scalar state variable; the vector-valued case follows component-wise. The technique is an extension of the derivation of the standard LCT (e.g. Ref.~\cite{smith2011introduction}) and utilises the following properties of the shifted Erlang distributions:
\begin{lemma}
    \label{lemma:derivatives}
    Let $g^k$, for $k\in\mathbb{N}$, be the $k^{\text{th}}$-order shifted Erlang distribution defined by \eqref{eq:Erlang_shifted}. For $s>\tau_{\mathrm{m}}$, the derivatives satisfy the following identities, with the corresponding boundary values at $s=\tau_{\mathrm{m}}$:
    \normalfont
    \begin{enumerate}
         \item $\begin{aligned}[t]
            \frac{d}{ds}g^1(s) = -\frac{1}{\Delta} g^1(s) \quad \text{and} \quad g^1(\tau_{\mathrm{m}}) = \frac{1}{\Delta},
                \end{aligned}$
        \vspace{0.2cm}
         \item $\begin{aligned}[t]
             \frac{d}{ds}g^k(s) = \frac{1}{\Delta} [ g^{k-1}(s) - g^{k}(s)]
             \quad \text{and} \quad g^k(\tau_{\mathrm{m}}) = 0, \; k>1.
                \end{aligned}$
    \end{enumerate}
\end{lemma}
\noindent
\emph{Proof.} 
    For $k=1$, the shifted Erlang kernel is exponential on $s>\tau_{\mathrm{m}}$, so property (1) follows directly.
    For $k>1$ and $s>\tau_{\mathrm{m}}$,
    \[
    g^k(s)
    =
    \frac{s-\tau_{\mathrm{m}}}{\Delta(k-1)}g^{k-1}(s).
    \]
    Differentiating the explicit expression gives
    \[
    \frac{d}{ds}g^k(s)
    =
    \left(\frac{k-1}{s-\tau_{\mathrm{m}}}-\frac{1}{\Delta}\right)g^k(s)
    =
    \frac{1}{\Delta}g^{k-1}(s)-\frac{1}{\Delta}g^k(s).
    \]
\hfill $\square$

These identities are the shifted analogues of those used in the classical LCT.
\begin{theorem}[Shifted Linear Chain Trick]
\label{thm:slct}
    Consider the scalar distributed-delay equation
    \begin{equation*}
        \dot{y}(t)
        =
        \mathcal{F}\left(
            y(t),
            \int_{-\infty}^{t} g^n(t-s)y(s)\,ds,
            t
        \right),
    \end{equation*}
    where $g^n$ is the $n^{\text{th}}$-order shifted Erlang kernel \eqref{eq:Erlang_shifted} with minimum delay $\tau_{\mathrm{m}}$ and scale parameter $\Delta$.
    An equivalent system is given by
    \begin{gather}
        \begin{split}
            \dot{y}(t) &= \mathcal{F}(y(t),x_n(t), t),\\
            \dot{x}_1(t) &= \frac{1}{\Delta} \left[y(t-\tau_{\mathrm{m}}) - x_1(t)\right],\\
            \dot{x}_k(t) &= \frac{1}{\Delta} \left[x_{k-1}(t) - x_k(t)\right], \; k = 2,3,\dots,n.
        \end{split}
    \end{gather}
\end{theorem}
\noindent
\emph{Proof.} 
    In a way similar to the LCT, we define auxiliary variables $x_k:\mathbb{R}\to\mathbb{R}$ for $k = 1, \dots, n$ as
    \begin{equation}
    \label{eq:definition}
     x_k(t) = \int_{-\infty}^{t} y(s) g^k (t - s) ds,
    \end{equation}
    so that $x_n$ is equal to the distributed-delay term in the shifted-distribution DDE. By the definition of $g^k$, we have $g^k (t - s) = 0$ for $s > t-\tau_{\mathrm{m}}$, giving
    \begin{equation}
        \label{eq:bounds}
        x_k(t) = \int_{-\infty}^{t-\tau_{\mathrm{m}}} y(s) g^k (t - s) ds.
    \end{equation}
    By the Leibniz integral rule,
    \begin{equation*}
        \dot{x}_k(t) = y(t-\tau_{\mathrm{m}}) g^k (\tau_{\mathrm{m}}) + \int_{-\infty}^{t-\tau_{\mathrm{m}}} y(s) \frac{d}{dt} g^k (t - s) ds.
    \end{equation*}
    Applying the properties from Lemma~\ref{lemma:derivatives},
    \begin{equation}
        \label{eq:properties}
        \dot{x}_k(t) = 
        \left\{
          \begin{aligned}
            & \frac{1}{\Delta} y(t-\tau_{\mathrm{m}}) - \frac{1}{\Delta} \int_{-\infty}^{t-\tau_{\mathrm{m}}} y(s) g^1(t-s) ds,  &  k=1, \\
            & \frac{1}{\Delta} \int_{-\infty}^{t-\tau_{\mathrm{m}}} y(s) \left[ g^{k-1}(t-s) - g^{k}(t-s)\right] ds,  & k > 1.
          \end{aligned}
          \right.
    \end{equation}
    Combining Equations~\eqref{eq:bounds} and \eqref{eq:properties} yields the final result.

    Thus, a distributed-delay DDE with a single $n^{\text{th}}$-order shifted Erlang kernel can be represented by a DDE with one constant delay coupled to $n$ auxiliary ODEs---provided the auxiliary histories satisfy the convolution definitions \eqref{eq:definition} on the initial history interval.
\hfill $\square$

%%%%%%%%%%%%%%%%%%%%%%%%%%%%%%%%%%%%%%%%%%%%%%%%%%%%
\section{Modelling ENSO with a shifted Erlang distributed delay}
\label{sec:ENSO}

As an application of this construction, we now consider a conceptual ENSO model, where the distributed delay represents oceanic wave transit times across the tropical Pacific. ENSO is a coupled ocean-atmosphere phenomenon in the tropical Pacific, characterised by irregular switching between years with anomalously warm El Ni{\~n}o conditions and anomalously cold La Ni{\~n}a conditions. The oceanic component is commonly measured by sea-surface temperature anomalies in the eastern equatorial Pacific Ocean, and the atmospheric component is captured by the Southern Oscillation, which is a pressure oscillation across the tropical Pacific. Figure~\ref{fig:enso_time_series} illustrates this coupled variability with the NINO3 index and the Southern Oscillation Index (SOI). The NINO3 index is the spatially averaged sea-surface temperature anomaly over $5^{\circ}$N to $5^{\circ}$S and $150^{\circ}$W to $90^{\circ}$W, and the SOI is defined as the air-pressure difference between Tahiti and Darwin. Peaks in NINO3, typically accompanied by negative SOI anomalies, correspond to El Ni{\~n}o events, while negative NINO3 anomalies and positive SOI anomalies correspond to La Ni{\~n}a events. El Ni{\~n}o events occur irregularly at intervals of approximately 4--7 years, yet their peaks are seasonally locked to occur between December and February \cite{wang2016nino}.

%%%%%%%%%%%%%%%%%%%%%%%%%%%%%%%%%%%%%%%%%%%%%%%%%%%%
\begin{figure}[!t]
  \hspace*{10mm}
  \includegraphics[scale=1.11]{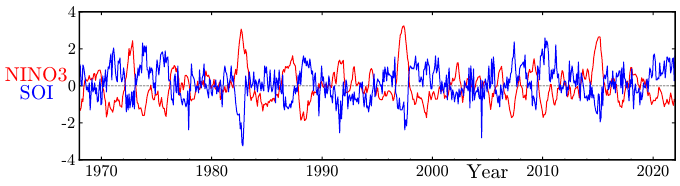}
\caption{Four-month running averages of the monthly NINO3 index (red) and SOI anomalies (blue) relative to the 1968--2022 long-term mean. Data are from NOAA.\protect\footnotemark}
\label{fig:enso_time_series}
\end{figure}
\footnotetext{NOAA National Centers for Environmental Information, ENSO monitoring data from \url{https://www.ncei.noaa.gov/access/monitoring/enso/}.}
%%%%%%%%%%%%%%%%%%%%%%%%%%%%%%%%%%%%%%%%%%%%%%%%%%%%

Key drivers of ENSO dynamics are delayed oceanic feedback loops associated with wave propagation across the equatorial Pacific Ocean \cite{tziperman1998locking, wang2016nino}. The Ghil-Zaliapin-Thompson model is a well-established conceptual scalar DDE model of ENSO with a single feedback loop where delay is assumed to be constant. It was introduced in \cite{ghil2008delay} and builds on earlier so-called delayed-action-oscillator models \cite{battisti1989interannual, suarez1988delayed, tziperman1998locking, tziperman1994nino}. The GZT model has since been studied in detail in \cite{BolducStAubinHumphries2026, keane2018chenciner, keane2019effect, keane2015delayed, keane2016investigating, krauskopf2014bifurcation, zaliapin2010delay}. It therefore provides a good starting point for considering the more realistic case of a distributed delay with a non-zero delay time. In particular, it provides a case study for applying the SLCT to a system with physically motivated distributed delay.

%%%%%%%%%%%%%%%%%%%%%%%%%%%%%%%%%%%%%%%%%%%%%%%%%%%%
\subsection{Delayed oceanic feedback}
\label{subsec:delayed_feedback}
The first task is to model the delay with an Erlang distribution.
The delayed feedback represented in the GZT model is associated with adjustments of the equatorial thermocline \cite{suarez1988delayed}.
This mechanism is illustrated schematically in Figure~\ref{fig:enso_diagram}, where the numbered labels indicate the main stages of the feedback loop.
The thermocline separates warmer surface water from colder deep water, and its depth in the eastern equatorial Pacific Ocean controls the strength of cold-water upwelling. In stage 1, a positive thermocline depth anomaly in the East, represented by $h(t)$, reduces upwelling and produces a warmer sea-surface temperature anomaly. In stage 2, ocean-atmosphere coupling modifies the trade winds and generates thermocline perturbations in the central equatorial Pacific Ocean that propagate westward as Rossby waves. In stage 3, reflection at the western boundary returns the signal eastward along the equator as a Kelvin wave.
The returning signal then changes the eastern thermocline depth, completing a delayed negative feedback loop.

This mechanism is generally idealised by a single constant delay, representing one characteristic (mean or effective) wave transit time across the Pacific basin \cite{battisti1989interannual, suarez1988delayed, tziperman1998locking}. However, physically, the feedback is not concentrated at one exact return time: wave reflection is not instantaneous, and spatial effects spread the returning signal over a range of arrival times \cite{boulanger1995propagation}. Thus, the delayed feedback is more naturally represented by a distributed delay. At the same time, the wave transit process has a non-zero minimum transit time, so a shifted delay distribution is appropriate, rather than a kernel that permits arbitrarily small delays. This is precisely the modelling situation in which shifted Erlang kernels and the SLCT are useful.

%%%%%%%%%%%%%%%%%%%%%%%%%%%%%%%%%%%%%%%%%%%%%%%%%%%%
\begin{figure}[t]
  \centering
  \includegraphics[scale=1.11]{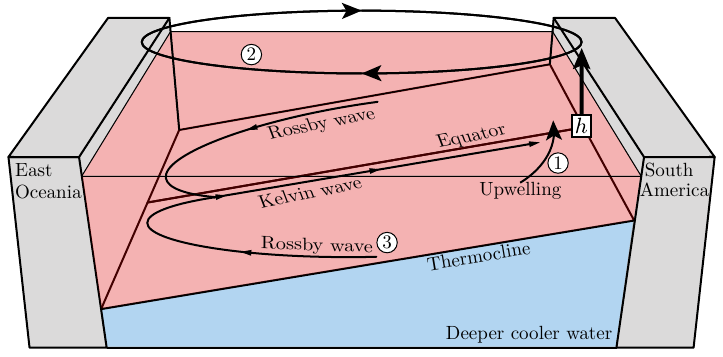}
  \caption{Schematic of the main ENSO feedback loop, with its stages represented by numbers (see the text). The variable $h$ represents thermocline depth anomalies in the eastern equatorial Pacific Ocean, and the water above and below the thermocline is coloured red and blue, respectively.}
    \label{fig:enso_diagram}
\end{figure}
%%%%%%%%%%%%%%%%%%%%%%%%%%%%%%%%%%%%%%%%%%%%%%%%%%%%

%%%%%%%%%%%%%%%%%%%%%%%%%%%%%%%%%%%%%%%%%%%%%%%%%%%%
\subsection{Distributed-delay GZT models}
\label{subsec:gzt_models}
In its standard constant-delay form, the GZT model is given by \cite{ghil2008delay}
\begin{equation}
\label{eq:gzt}
\dot{h}(t)
=
-a \tanh\left(\kappa h(t-\tau)\right)
+
c \cos(\omega t),
\end{equation}
where $h$ is the thermocline depth anomaly in the eastern equatorial Pacific Ocean. The first term represents nonlinear delayed negative feedback from the returning oceanic wave signal, with feedback strength $a$ and ocean-atmosphere coupling parameter $\kappa$. The delay $\tau$ is measured in years, and the second term represents annual seasonal forcing of strength $c$ so that $\omega=2\pi$. In spite of its simplicity, the model reproduces qualitative features associated with ENSO, including irregular peaks in thermocline-depth anomalies and seasonal locking \cite{zaliapin2010delay}.
Throughout, we fix $\kappa=11$ and $a=1$, which are established standard parameter values.

The constant delay in \eqref{eq:gzt} corresponds to the limiting case in which the returning wave signal is assumed to arrive after a unique and fixed transit time. To allow for a more realistic spread of arrival times, we now introduce a distributed feedback term with shifted Erlang kernel $g^n$ to obtain the distributed-delay GZT model
\begin{equation}
\dot{h}(t)
=
-a \tanh\left(
\kappa
\int_{-\infty}^{t-\tau_{\mathrm{m}}} h(s)g^n(t-s)\,ds
\right)
+
c \cos(2 \pi t).
\label{eq:GZT_ddde}
\end{equation}
The delayed feedback is interpreted as a wave-transit process with a characteristic (mean) return time \cite{battisti1989interannual, suarez1988delayed, vanoldenborgh1999tracking}. In this setting, $\tau_{\mathrm{m}}$ represents the earliest possible wave return, while the main contribution is expected near a later characteristic travel time. In the absence of measured data on wave returns across the Pacific Ocean, we construct an empirical Rossby--Kelvin delay kernel from the wave-speed estimates of Boulanger and Menkes \cite{boulanger1995propagation}; see Appendix~\ref{app:kernel} for details. This empirical kernel, shown in Figure~\ref{fig:fitted_n3_distribution}, constitutes a physically motivated estimate of the transit time distribution.
It has an effective delay of $0.593$ years, which is close to the commonly cited value of six months \cite{battisti1989interannual, tziperman1997controlling}.

We seek to represent this empirical transit-time distribution by a single low-order shifted Erlang kernel. The case $n=1$ is not appropriate physically, because the kernel is maximal immediately at the minimum delay $\tau_{\mathrm{m}}$, as illustrated in Figure~\ref{fig:delay_kernels}(a). We find that, among the low-order kernels in Figure~\ref{fig:delay_kernels}, the case $n=3$ gives the most suitable compromise: the case $n=2$ does not capture the more gradual onset of the empirical distribution, while kernels of order $n=4$ and higher are less effective at representing its skewness. Figure~\ref{fig:fitted_n3_distribution} also shows the selected shifted Erlang kernel with $n=3$, minimum delay $\tau_{\mathrm{m}}=0.35$ and scale parameter $\Delta = 1/15 = 0.0\bar{6}$. It captures key properties well, namely the delayed onset and central mass of the estimated transit-time distribution. From now on, we fix $n=3$ when studying \eqref{eq:GZT_ddde}. There are some differences in the height of the peak and the decay for longer delays. However, keeping in mind that the empirical distribution is in itself only an estimate of the actual distribution of wave arrival times, we judge that this shifted Erlang distribution provides a suitable delay distribution at a level suited to the conceptual GZT model. While its empirical width is given by $\Delta=0.0\bar{6}$, the advantage of the Erlang representation is that we are able to ascertain the effect of changing $\Delta$ from $\Delta=0$, the case of constant delay, all the way to $\Delta=0.0\bar{6}$ and beyond.

%%%%%%%%%%%%%%%%%%%%%%%%%%%%%%%%%%%%%%%%%%%%%%%%%%%%
\begin{figure}[!t]
  \centering
  \includegraphics[scale=1.11]{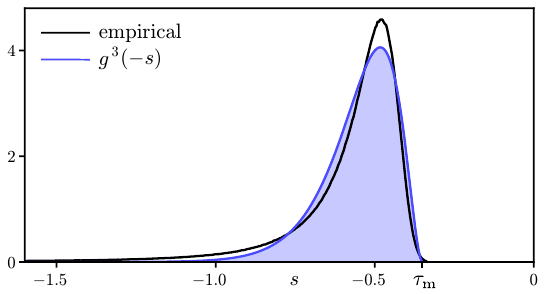}
    \caption{Empirical Rossby--Kelvin delay kernel (solid black), constructed from the wave-speed estimates of Boulanger and Menkes \cite{boulanger1995propagation} as described in Appendix~\ref{app:kernel}, together with the shifted Erlang distribution (blue) with $n=3$, $\tau_{\mathrm{m}} = 0.35$ and $\Delta = 0.0\bar{6}$.}
    \label{fig:fitted_n3_distribution}
\end{figure}
%%%%%%%%%%%%%%%%%%%%%%%%%%%%%%%%%%%%%%%%%%%%%%%%%%%%

\par
Applying Theorem~\ref{thm:slct} to \eqref{eq:GZT_ddde} gives
\begin{gather}
\begin{split}
\label{eq:gzt_slct}
\dot{h}(t) &= -a \tanh\left(\kappa h_n(t)\right) + c \cos(2\pi t),\\
\dot{h}_1(t) &= \frac{1}{\Delta} \left[ h(t-\tau_{\mathrm{m}})-h_1(t) \right],\\
\dot{h}_k(t) &= \frac{1}{\Delta} \left[ h_{k-1}(t)-h_k(t) \right], \;k=2,\dots,n.
\end{split}
\end{gather}
This DDE with the constant delay $\tau_{\mathrm{m}}$ and auxiliary variables $h_1,\dots,h_n$ avoids direct numerical treatment of the convolution integral and renders the distributed-delay GZT model amenable to bifurcation analysis.

%%%%%%%%%%%%%%%%%%%%%%%%%%%%%%%%%%%%%%%%%%%%%%%%%%%%
\section{Bifurcation analysis of the GZT model with delay distribution} 
\label{subsec:results}

A useful tool for studying the GZT model is to compute maximum maps in the plane of delay versus forcing strength \cite{keane2018chenciner, keane2019effect, keane2015delayed, keane2016investigating, zaliapin2010delay}. A maximum map is obtained by integrating the model at each point of a suitably fine parameter grid and plotting the maximum value of $h(t)$ over the resulting time series, after transients have been discarded. This is done with a sweeping technique, in which a selected parameter is increased or decreased in small steps and the final history from each computation is used as the initial history for the next. In this way, we obtain a numerical overview of the dynamics with hints of the organising bifurcation structure of the model. Maximum maps provide a compact numerical picture of resonance tongues, dynamical transitions, and the organising bifurcation structure of the model. The advantage of the SLCT formulation \eqref{eq:gzt_slct} is that integration is fast because computing the distributed delay at each time step is not required and, more importantly, we are able to compute bifurcation curves of interest. We use maximum maps in what follows to compare shifted distributed-delay models across kernel widths, with the narrow-kernel limit providing the constant-delay reference point.

%%%%%%%%%%%%%%%%%%%%%%%%%%%%%%%%%%%%%%%%%%%%%%%%%%%%
\subsection{The case of an extremely narrow distribution}
\label{subsec:narrow}

%%%%%%%%%%%%%%%%%%%%%%%%%%%%%%%%%%%%%%%%%%%%%%%%%%%%
\begin{figure}[!t]
  \centering
  \includegraphics[scale=1.2]{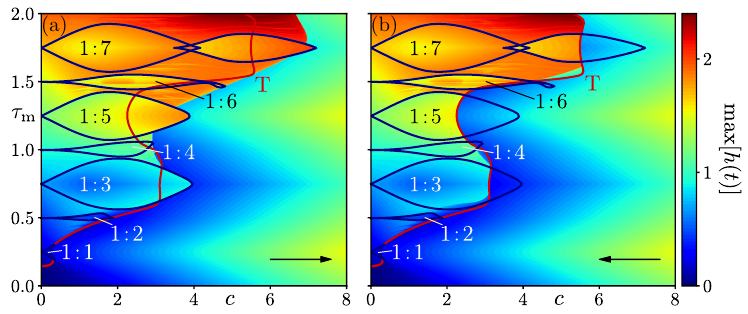}
    \caption{Maximum maps and bifurcation curves in the parameter plane $(c, \tau_{\mathrm{m}})$ of the SLCT formulation~\eqref{eq:gzt_slct} for $g^3(s)$ with $\Delta=10^{-3}$. Panel~(a) shows an upward sweep and panel~(b) a downward sweep in $c$, as is indicated by the black arrows. Also shown are selected curves SN of saddle-node bifurcations of periodic orbits (blue) and the curve T of torus bifurcations (red); resonance tongues with locking ratios $1\!:\!q$ for $q=1,\dots,7$ are labelled.}
    \label{fig:narrow_kernel_maps}
\end{figure}
%%%%%%%%%%%%%%%%%%%%%%%%%%%%%%%%%%%%%%%%%%%%%%%%%%%%

Figure~\ref{fig:narrow_kernel_maps} shows, as our starting point, maximum maps in the $(c,\tau_{\mathrm{m}})$-plane of \eqref{eq:gzt_slct} for $\Delta = 10^{-3}$, corresponding to the narrow kernel illustrated in Figure~\ref{fig:delay_kernels}(c) which represents the constant delay case. The forcing amplitude $c$ is swept upward in panel~(a) and downward in panel~(b). Curves of saddle-node bifurcations of periodic orbits (SN) and a curve of torus bifurcations (T) have been computed in \texttt{DDE-BifTool} and are overlaid in blue and red, respectively. The curves SN form the boundaries of resonance tongues, which are regions of frequency-locked periodic dynamics. The curve T marks the loss of stability of periodic orbits to dynamics on an invariant torus. For clarity, only a representative set of resonance tongues with locking ratios $1\!:\!q$ for $q=1,\dots,7$ is shown. The overlaid curves organise the visible transitions between frequency-locked periodic dynamics and more complex quasiperiodic and chaotic behaviour. The sharp boundary visible in the upward sweep is associated with a `fold of tori', which is a complicated global bifurcation mechanism that cannot be continued as a curve \cite{bolduc2026resonance, keane2018chenciner}; we will discuss it in somewhat more detail in Section~\ref{subsec:effect}.

For such a narrow kernel, the shifted Erlang distribution is concentrated so close to $\tau_{\mathrm{m}}$ that it is effectively a delta function. In this limit, the distributed-delay term becomes the constant-delay feedback term, so that \eqref{eq:gzt_slct} with $\Delta=10^{-3}$ closely approximates the constant-delay GZT model~\eqref{eq:gzt}. Indeed, Figure~\ref{fig:narrow_kernel_maps} reproduces the familiar bifurcation structure with $\tau=\tau_{\mathrm{m}}$ reported in previous studies \cite{keane2018chenciner, keane2017climate}, and we use this case as the `ground truth' for subsequent comparisons.

Figure~\ref{fig:narrow_kernel_maps} also demonstrates a practical advantage of the SLCT formulation over treating the distributed feedback via a convolution integral, which becomes numerically very sensitive to such sharply concentrated kernels. Even in the delta-function-like limit, the auxiliary-variable formulation remains well behaved, allowing us to investigate the effect of delay distribution for any value of the scale parameter $\Delta$.

%%%%%%%%%%%%%%%%%%%%%%%%%%%%%%%%%%%%%%%%%%%%%%%%%%%%
\begin{figure}[!t]
  \centering
  \includegraphics[scale=1.2]{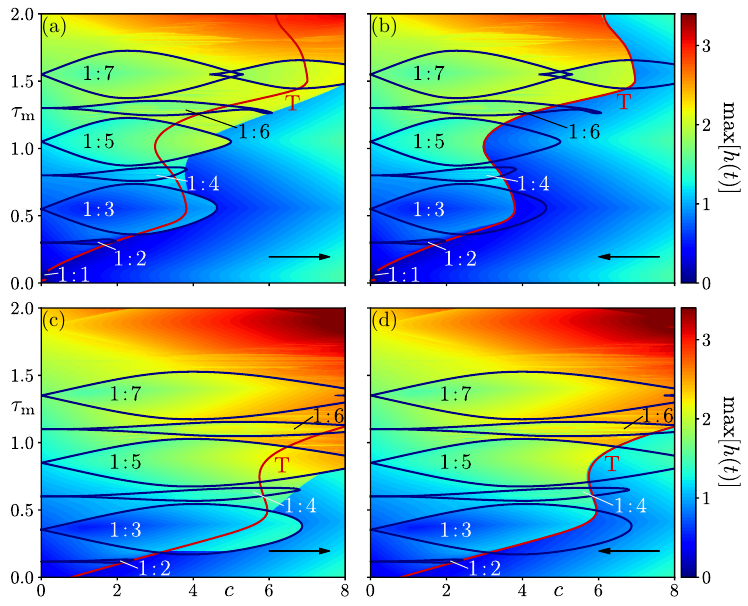}
    \caption{Maximum maps and bifurcation curves in the $(c, \tau_{\mathrm{m}})$-plane of the SLCT formulation~\eqref{eq:gzt_slct} for $g^3(s)$ with $\Delta=0.0\bar{6}$ in panels~(a) and~(b), and $\Delta=0.1\bar{3}$ in panels~(c) and~(d), presented as in Figure~\ref{fig:narrow_kernel_maps}.}
    \label{fig:n3_maps}
\end{figure}
%%%%%%%%%%%%%%%%%%%%%%%%%%%%%%%%%%%%%%%%%%%%%%%%%%%%

%%%%%%%%%%%%%%%%%%%%%%%%%%%%%%%%%%%%%%%%%%%%%%%%%%%%
\subsection{Wider distributions and scalings for fair comparison}
\label{subsec:rescaling}

Figure~\ref{fig:n3_maps} shows, in the format of Figure~\ref{fig:narrow_kernel_maps} and over the same ranges, the 
maximum maps and bifurcation curves in the $(c,\tau_{\mathrm{m}})$-plane for two wider shifted Erlang kernels. Panels (a) and (b) of Figure~\ref{fig:n3_maps} are for $\Delta=0.0\bar{6}$, which is the value that we selected to match the empirical Rossby--Kelvin delay kernel reasonably well, as illustrated in Figure~\ref{fig:fitted_n3_distribution}. At first glance, one notices that the overall structure of the upward and downward sweeps is very similar to that in Figure~\ref{fig:narrow_kernel_maps}. The main difference in Figure~\ref{fig:n3_maps}(a) and (b) is a shift of the shown resonance tongue structure towards lower values of $\tau_{\mathrm{m}}$, as well as its `stretching' in $c$. For the substantially wider kernel with $\Delta=0.1\bar{3}$ in panels~(c) and (d), these two effects are considerably more pronounced.

Overall, Figure~\ref{fig:n3_maps} clearly shows that it is not suitable to represent the results for increasing scale parameter $\Delta$ in terms of the minimum delay $\tau_{\mathrm{m}}$. In light of the discussion of shifted Erlang distributions in Section~\ref{subsec:erlang_kernels}, it is more natural to consider the effective delay $\tau_{\mathrm{eff}} = \tau_{\mathrm{m}} + n\Delta$ from \eqref{eq:taueff} instead, because it is the expected or mean value of the distribution.

The observed stretching in the forcing strength $c$ is arguably more surprising. It can be understood and remedied by realising that increasing $\Delta$ also changes the strength of the feedback term relative to $c$. This effect can be quantified by examining how the distributed-delay operator filters the feedback signal. Since convolution with the delay kernel is a linear time-invariant operation, each sinusoidal component is multiplied by the Fourier transform of the kernel \cite[Chapter~2]{oppenheim1997signals}. Thus, the shifted Erlang kernel acts as a frequency-dependent filter on the delayed feedback. The Fourier response of the shifted $n^{\text{th}}$-order Erlang kernel is
\begin{equation*}
    G(i\omega) = e^{-i\omega\tau_{\mathrm{m}}}(1+i\omega\Delta)^{-n},
\end{equation*}
and the modulus of this response gives the gain at frequency $\omega$ as
\begin{equation*}
    |G(i\omega)|
    =
    (1+(\omega\Delta)^2)^{-n/2}.
\end{equation*}
Thus, broadening an Erlang kernel weakens the delayed feedback signal.

Since the solutions of the GZT model arise through a competition between feedback and annual forcing, the leading effect of the distributed delay is captured by the gain $|G(2\pi i)|$ at the forcing frequency $\omega=2\pi$. We therefore define the corresponding attenuation factor
\begin{equation}
    \label{eq:alpha}
    \alpha(n,\Delta)
    =
    \frac{1}{|G(2\pi i)|}
    =
    (1+(2\pi\Delta)^2)^{n/2},
\end{equation}
to obtain the rescaled feedback strength
\begin{equation}
\label{eq:a_rescale}
    \tilde{a} = \frac{a}{\alpha(n,\Delta)}.
\end{equation}
Replacing $a$ with $\tilde{a}$ in the distributed-delay GZT model \eqref{eq:GZT_ddde} and its SLCT formulation \eqref{eq:gzt_slct} allows us to compare the bifurcation and resonance structure like-for-like in the $(c,\tau_{\mathrm{m}})$-plane for any value of the scale parameter $\Delta \geq 0$; note, in particular, that this includes the limiting case of constant-delay feedback because $\tau_{\mathrm{eff}} \to \tau_{\mathrm{m}}$ and $\tilde{a} \to a$ as $\Delta \to 0$.

%%%%%%%%%%%%%%%%%%%%%%%%%%%%%%%%%%%%%%%%%%%%%%%%%%%%
\begin{figure}[!t]
  \centering
  \includegraphics[scale=1.2]{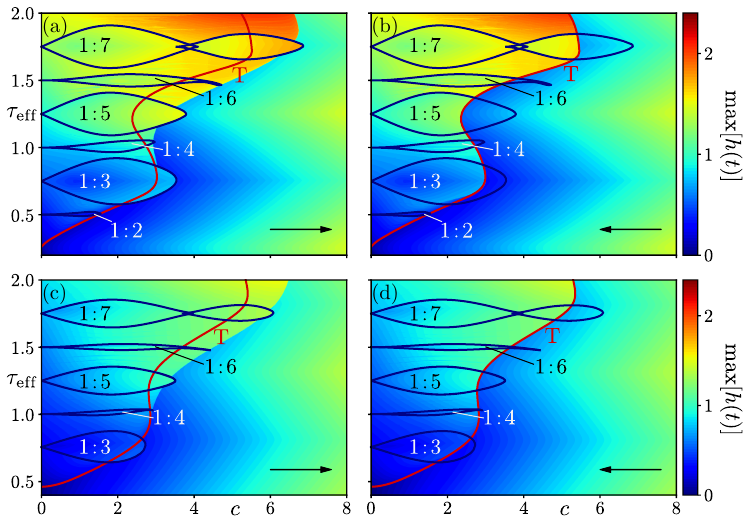}
  \caption{Maximum maps and bifurcation curves in the $(c,\tau_{\mathrm{eff}})$-plane of the SLCT formulation~\eqref{eq:gzt_slct} for $g^3(s)$ with the rescaled feedback strength $\tilde{a}$ from \eqref{eq:a_rescale}, and with $\Delta=0.0\bar{6}$ and $\tau_{\mathrm{eff}} \geq 0.2$ in panels~(a) and~(b), and $\Delta=0.1\bar{3}$ and $\tau_{\mathrm{eff}} \geq 0.4$ in panels~(c) and~(d). Compare with Figure~\ref{fig:narrow_kernel_maps}.}
    \label{fig:n3_maps_a_rescaled}
\end{figure}
%%%%%%%%%%%%%%%%%%%%%%%%%%%%%%%%%%%%%%%%%%%%%%%%%%%%

%%%%%%%%%%%%%%%%%%%%%%%%%%%%%%%%%%%%%%%%%%%%%%%%%%%%
\subsection{Effect of delay distributions for the GZT model}
\label{subsec:effect}

Figure~\ref{fig:n3_maps_a_rescaled} shows the maximum maps with bifurcation curves SN and T from Figure~\ref{fig:n3_maps}, but now in the $(c,\tau_{\mathrm{eff}})$-plane of the SLCT formulation~\eqref{eq:gzt_slct} for $n=3$ with the rescaled feedback strength $\tilde{a}$ from \eqref{eq:a_rescale}. Note that the requirement that $\tau_{\mathrm{m}}$ be positive implies that $\tau_{\mathrm{eff}} \geq n\Delta$; hence, panels~(a) and~(b) for $\Delta=0.0\bar{6} = 1/15$ show the range $\tau_{\mathrm{eff}} \in [0.2, 2.0]$, and panels~(c) and~(d) for $\Delta=0.1\bar{3} = 2/15$ the range $\tau_{\mathrm{eff}} \in [0.4, 2.0]$. The attenuation of the feedback strength is accounted for by the quantity $\alpha(n,\Delta)$ from \eqref{eq:alpha}, which has the values $\alpha(3, 0.0\bar{6}) \approx 1.27$ in panels~(a) and~(b), and $\alpha(3, 0.1\bar{3}) \approx 2.22$ in panels~(c) and~(d).

Comparison with Figure~\ref{fig:narrow_kernel_maps} shows that the representation in Figure~\ref{fig:n3_maps_a_rescaled} indeed avoids scaling issues when $\Delta$ is changed. Hence, a like-for-like comparison of the bifurcation structure of the distributed-delay GZT model for different values of $\Delta$ is now possible, and any differences can be attributed to the influence of the delay distribution itself. Overall, we observe that the bifurcation and resonance structure in Figure~\ref{fig:n3_maps_a_rescaled} representing the limiting case of constant delay $\tau_{\mathrm{eff}} = \tau_{\mathrm{m}}$ does not change drastically when $\Delta$ is increased to the `realistic' value of $0.0\bar{6}$ as in Figure~\ref{fig:narrow_kernel_maps}(a) and~(b) and then beyond to $\Delta=0.1\bar{3}$ as in panels~(c) and~(d). There are some differences, however.

First of all, the maxima of $h(t)$ recorded by the maximum maps become smaller, as is evidenced by the reduced range of the colour map. Secondly, the resonance tongues become narrower with increasing $\Delta$. This effect is most pronounced for the $1\!:\!1$ and $1\!:\!2$ resonance tongues, which emerge along the $c$-axis at $\tau_{\mathrm{eff}} = 0.25$ and $\tau_{\mathrm{eff}} = 0.5$, respectively. Moreover, for $\Delta=0.0\bar{6}$ as in panels (a) and (b) of Figure~\ref{fig:n3_maps_a_rescaled}, the $1\!:\!1$ resonance tongue is no longer in the shown frame (where $\tau_{\mathrm{m}} \geq 0$), and the $1\!:\!2$ resonance tongue has also disappeared for $\Delta=0.1\bar{3}$ in panels (c) and (d). Finally, the arguably largest difference concerns the position and shape of the torus bifurcation curve T of the forcing-dominated periodic orbit. Importantly, this curve forms the left boundary of a sizeable region of coexistence between an attracting invariant torus and the forcing-dominated periodic orbit. The right boundary of this region is identified by a sudden colour change in the upward sweeps in Figure~\ref{fig:narrow_kernel_maps}(a) and Figure~\ref{fig:n3_maps_a_rescaled}(a) and~(c), which represents a sudden change in the size of the maximum of $h(t)$. Mathematically, this represents the locus where an attracting and a saddle invariant torus `collide' and disappear. This `fold bifurcation of tori' is a complicated phenomenon that involves the break-up of the two tori and an accumulation of folds of resonance tongues; see \cite{bolduc2026resonance, keane2018chenciner}.

These changes of the bifurcation and resonance structure are consistent with the fact that, for fixed scale parameter $\Delta$, the kernel width represents a larger proportion of the effective delay. The relative spread of the delayed feedback is therefore greater, leading to a stronger departure from the constant-delay limit. On the other hand, comparison between Figure~\ref{fig:narrow_kernel_maps} and Figure~\ref{fig:n3_maps_a_rescaled}(a) and~(b) demonstrates that the effect of delay distribution with a `realistic' width does not lead to major changes. Rather, the differences are quantitative and somewhat subtle.

%%%%%%%%%%%%%%%%%%%%%%%%%%%%%%%%%%%%%%%%%%%%%%%%%%%%
\section{Summary and outlook}
\label{sec:conclusion}

We formulated the SLCT for a general scalar DDE with distributed delay given by a shifted Erlang kernel, and presented the resulting finite auxiliary-variable chain with one constant delay. While it is still a DDE, this reduction avoids the evaluation of the kernel integral --- making it much more efficient for simulations and allowing one to employ numerical continuation techniques for constant-delay DDEs. We demonstrated this for the GZT ENSO model, which is a scalar DDE with a delayed feedback term and seasonal forcing that describes the evolution of temperature deviations in the eastern equatorial Pacific Ocean. More specifically, we introduced delay distribution into this established DDE model by considering a single Erlang kernel that matches well an empirical `best estimate' of wave travel times across the Pacific. The reduction of the distributed-delay GZT model via the SLCT allowed us to determine the influence of the width of the distribution on the bifurcation and resonance structure in the parameter plane of forcing strength versus effective delay. It is important to realise that the width of the delay weakens the strength of the distributed-feedback term relative to other parameters, and we showed how the feedback strength in the GZT model can be rescaled to counter this effect.

For ease of exposition and in light of the example we considered, the SLCT was presented here for a scalar DDE with a single shifted Erlang kernel. However, it can be employed much more widely --- effectively, for any DDE model that features delay distribution with a minimum delay. Namely, the SLCT generalises readily to non-scalar DDEs by applying it component-wise to vectors of variables. Moreover, it is possible to consider weighted sums of shifted Erlang kernels, each with its own minimum delay and scale parameter. This provides a flexible way of representing any delay distribution of interest, including more complicated ones, for example, those with two or even more maxima. When the SLCT is applied to such a sum, it results in a DDE with the same auxiliary-variable structure and as many constant delays as there are different minimum delays (which all have to be positive but need not be the same). Since the weights, minumum delays and associated scale parameters can be varied in any way in this reduced representation, including as continuation parameters, their influence on observed dynamics can be examined. However, as we discussed, this requires determining how the associated feedback strengths need to be rescaled relative to one another and to non-feedback terms to ensure a fair comparison.

%%%%%%%%%%%%%%%%%%%%%%%%%%%%%%%%%%%%%%%%%%%%%%%%%%%%
\section*{Acknowledgements}
We thank Tyler Cassidy and Sue Ann Campbell for helpful discussions and for pointing out useful background literature. The research of B.K. was supported by Royal Society Te Ap\={a}rangi Marsden Fund grant \#19-UOA-223.

%%%%%%%%%%%%%%%%%%%%%%%%%%%%%%%%%%%%%%%%%%%%%%%%%%%%
\appendix
%%%%%%%%%%%%%%%%%%%%%%%%%%%%%%%%%%%%%%%%%%%%%%%%%%%%
\section{Empirical Rossby--Kelvin delay kernel}
\label{app:kernel}

Under the common conceptual modelling assumption that the wave perturbation due to the interaction with the atmosphere originates halfway between the eastern and western boundaries, the westward Rossby-wave path has length $d/2$, while the reflected Kelvin wave traverses the full distance $d$ of the equatorial Pacific. Hence, the corresponding return time is
\begin{equation}
    \label{eq:round_trip_time}
    T = \frac{d}{2 c_{\mathrm{R}}} + \frac{d}{c_{\mathrm{K}}},
\end{equation}
where $c_{\mathrm{R}}$ and $c_{\mathrm{K}}$ denote the magnitudes of the Rossby and Kelvin phase speeds. The phase-speed estimates are taken from Boulanger and Menkes \cite{boulanger1995propagation}, who report
\begin{eqnarray}
\label{eq:speedK}
  c_{\mathrm{K}} \! \! \! \! &=& \!\!\!\! 2.3 \pm 0.3 \,\mathrm{m\,s^{-1}}, \ \mathrm{and}\\
  \label{eq:speedsR}
    |c_{\mathrm{R1}}| \! \!\!\! &=& \! \!\!\! 1.1 \pm 0.2 \,\mathrm{m\,s^{-1}},
    |c_{\mathrm{R2}}| = 0.8 \pm 0.3 \,\mathrm{m\,s^{-1}},
    |c_{\mathrm{R3}}| = 0.5 \pm 0.5 \,\mathrm{m\,s^{-1}}
\end{eqnarray}
for the first three Rossby modes. The latter are combined to give $c_{\mathrm{R}}$ by using the western-boundary reflection coefficients reported in Appendix~C of \cite{boulanger1995propagation}; for the meridionally bounded western-boundary case, they are
\begin{equation*}
    (a_1,a_2,a_3)
    \approx
    (0.41,\ 0.60,\ 0.13).
\end{equation*}
We take the reflected contribution of each of the three modes to be proportional to the squared reflection coefficient $a_n^2$, which gives the normalised Rossby-mode weights
\begin{equation}
  \label{eq:weights}
    (w_1,w_2,w_3)
    \approx
    (0.308,\ 0.661,\ 0.031).
\end{equation}
Thus, the reference delay kernel is dominated by the second Rossby mode, with a substantial first-mode and a small third-mode contribution.

An empirical delay kernel is then constructed by Monte Carlo sampling, where, for each sample, a Rossby mode is selected according to the weights given by \eqref{eq:weights}. The Kelvin phase speed and the phase speed of the selected Rossby mode are sampled from truncated normal distributions with the means and standard deviations given by \eqref{eq:speedK} and \eqref{eq:speedsR}. These sampled speeds are combined by using \eqref{eq:round_trip_time} to obtain a Rossby--Kelvin return time over the $110^\circ$ longitude corridor used in \cite{boulanger1995propagation}. The resulting travel times are then rescaled by the factor $150/110$ to represent propagation between the eastern and western boundaries of the tropical Pacific. The density of these rescaled Monte Carlo travel times defines the empirical delay kernel shown in Figure~\ref{fig:fitted_n3_distribution}.

%%%%%%%%%%%%%%%%%%%%%%%%%%%%%%%%%%%%%%%%%%%%%%%%%%%%
%% \bibliographystyle{abbrv}
%% \bibliography{SKK_SLCT_references}

%%%%%%%%%%%%%%%%%%%%%%%%%%%%%%%%%%%%%%%%%%%%%%%%%%%%

\end{document}